\documentclass[10pt]{amsart}
\usepackage{amssymb}
\usepackage{amsmath}
\usepackage{amsthm}
\usepackage{latexsym}
\usepackage{graphicx}
\usepackage{hyperref}
\newcommand{\Z}{{\mathbb Z}}
\newcommand{\R}{{\mathbb R}}
\newcommand{\C}{{\mathbb C}}

\newcommand{\T}{{\mathbb T}}

\newtheorem{lemma}{Lemma}[section]
\newtheorem{theorem}[lemma]{Theorem}

\newtheorem{remark}[lemma]{Remark}
\newtheorem{proposition}[lemma]{Proposition}

\newtheorem{definition}[lemma]{Definition}

\newtheorem{question}[lemma]{Question}

\newcommand{\be}{\begin{equation}}
\newcommand{\ee}{\end{equation}}
\newcommand{\ul}{\underline}
\newcommand{\ol}{\overline}
\newcommand{\ti}{\tilde}

\newcommand{\spr}[2]{\left\langle #1 , #2 \right\rangle}

\newcommand{\E}{\mathrm{e}}
\newcommand{\I}{\mathrm{i}}

\newcommand{\tr}{\mathrm{tr}}
\newcommand{\im}{\mathrm{Im}}

\newcommand{\eps}{\varepsilon}

\numberwithin{equation}{section}

\begin{document}

\title{Periodic and limit-periodic discrete Schr\"odinger
 operators}

\author[H.\ Kr\"uger]{Helge Kr\"uger}
\address{Mathematics 253-37, Caltech, Pasadena, CA 91125}
%\address{Erwin Schr\"odinger Institute, Boltzmanngasse 9, A-1090 Vienna, Austria}
% \address{Department of Mathematics, Rice University, Houston, TX~77005, USA}
\email{\href{helge@caltech.edu}{helge@caltech.edu}}
\urladdr{\href{http://www.its.caltech.edu/~helge/}{http://www.its.caltech.edu/~helge/}}

\thanks{H.\ K.\ was supported by the Simons Foundation.}

\date{\today}

%\subjclass[2000]{Primary 81Q10; Secondary 37D25}

\begin{abstract}
 The theory of discrete periodic and limit-periodic Schr\"odinger operators
 is developed. In particular, the Floquet--Bloch decomposition
 is discussed. Furthermore, it is shown that an arbitrarily small potential
 can add a gap for even periods. In dimension two, it is shown
 that for coprime periods small potential terms don't add gaps
 thus proving a Bethe--Sommerfeld type statement. Furthermore
 limit-periodic potentials whose spectrum is an interval
 are constructed.
\end{abstract}

\maketitle

%%%%%%%%%%%%%%%%%%%%%%%%%%%%%%%%%%%%%%%%%%%%%%%%%%%%%%%%%%%%%%%%%%%%%
%
%
%

\section{Introduction}

My aim in this paper is two fold. The first three sections
discuss the basic theory of discrete periodic Schr\"odinger
operators on $\Z^d$. The reason for writing this is that
there are good reference on $\Z$ \cite{szego}, \cite{teschl}
and for continuum operators \cite{rs4}, \cite{skri},
but as far as I know no reference on $\Z^d$ for $d\geq 2$.
Then in the second part, I present new results on
Schr\"odinger operators on $\Z^2$:
\begin{enumerate}
 \item Theorem~\ref{thm:BSConj} shows that if the periods
  are coprime, then small enough perturbations do not
  add gaps in the spectrum.
 \item Theorem~\ref{thm:ex2p} valid in any dimension shows
  that there exist arbitrarily small 
  perturbations of even periods adding
  one gap in the spectrum.
 \item Finally Theorem~\ref{thm:lpinterval} exhibits a
  large class of limit-periodic potentials whose spectrum
  is an interval.
\end{enumerate}
The main ingredient in the proof of (i) and (iii)
is Theorem~\ref{thm:coprimeevsimple} which asserts that
any energy can be an eigenvalue of multiplicity
at least two for at most finitely many operators in
the Floquet--Bloch decomposition.

Both (i) and (iii) are phenomena appearing in dimension
two. In dimension one, one generally has
gaps, see Avila \cite{a1},
Avron--Simon \cite{as1}, Damanik--Gan \cite{dg1}, 
\cite{dg2}, Kr\"uger--Gan \cite{gk1}.
It is an interesting task to prove statements
analog to (i) and (iii) in dimensions three and
higher.

I consider (i) an analog of the Bethe--Sommerfeld
conjecture for continuous Schr\"odinger operators.
This conjecture states that for $d \geq 2$ and any
periodic function $V: \R^d \to\R$ the spectrum of
the operator
\be
 - \sum_{j=1}^{d} \frac{\partial^2}{\partial x_j^2} + V
\ee
only contains finitely many gaps. This conjecture
has been solved completely by Parnovski in \cite{p1},
see also \cite{bp1}, \cite{ps1} for more recent work. 
For some earlier work see the books \cite{karpe} and
\cite{skri}.
In \cite{kl1}, \cite{kl2}, \cite{kl3}, \cite{kl4},
Karpeshina and Lee have derived analogous statements
to (iii) in the continuum setting. In fact using
their KAM-type methods, Karpeshina and Lee are able
to prove absolutely continuous spectrum.

A difference between the continuum case and
the discrete case considered here, is that the
discrete statement depends on the underlying period
lattice as (ii) shows.

I have included Questions \ref{ques:idslp},
\ref{ques:coprime}, and \ref{ques:gaps} in order
to highlight some problems that would allow us to
gain further understanding of higher dimensional
operators. These questions do not address how
to construct operators with pure-point spectrum,
since this has already been solved by P\"oschel
in \cite{poe1}.

%%%%%%%%%%%%%%%%%%%%%%%%%%%%%%%%%%%%%%%%%%%%%%%%%%%%%%%%%%%%%%%%%%%%%
%
%
%

\section{Periodic discrete Schr\"odinger operators}
\label{sec:perschroe}

In this section, I discuss the spectral theory of discrete
periodic Schr\"odinger operators. Since, I am unaware of a source
of this, the discussion is somewhat detailed. Discussions
in the continuous case can be found in Reed--Simon \cite{rs4},
Skriganov \cite{skri}.

\subsection{Periodic functions}

We recall that given periods $\ul{p} = \{p_j\}_{j=1}^{d} \in (\Z_+)^d$,
a function $f:\Z^d\to\C$ is called {\em $\ul{p}$-periodic} if
\be
 f(\ul{n} + p_j b_j) = f(\ul{n})
\ee
for all $1\leq j\leq d$ and $\ul{n}\in\Z^d$, where
$b_j$ denotes the standard basis of $\Z^d$. Define
\be
 \mathbb{B}=\times_{j=1}^{d} 
  \left\{0, \frac{1}{p_j}, \dots, \frac{p_j - 1}{p_j}\right\}
   \subseteq \T^d.
\ee
We denote $\T=\R/\Z$ and $e(x) = \E^{2\pi \I x}$ as usual.
Define the Fourier coefficients $\hat{f}:\mathbb{B}\to\C$
of a $\ul{p}$-periodic function $f$ by
\be
 \hat{f}(\ul{k}) = \frac{1}{P}\sum_{\ul{n}\in\mathbb{B}} f(\ul{n}) 
  e\left(- \ul{k} \cdot \ul{n}\right),\quad
  \ul{k}\cdot\ul{n}=\sum_{j=1}^{d} k_j n_j
\ee
with $P =\prod_{j=1}^{d} p_j$.
One easily checks that
\be
 f(\ul{n}) = \sum_{\ul{k}\in\mathbb{B}} \hat{f}(\ul{k}) 
  e\left(\ul{k} \cdot \ul{n}\right).
\ee
Recall that for a function $u\in\ell^1(\Z^d)$, its Fourier transform
$Fu = \hat{u}:\T^d\to\C$ is given by $Fu(\ul{x}) = \hat{u}(\ul{x})=\sum_{\ul{n}\in\Z^d} e(-\ul{n}\cdot\ul{x}) u(\ul{n})$.
Then $F$ is extended to $\ell^2(\Z^d)\to\ell^2(\Z^d)$
by continuity. Plancherel's identity shows that
this map is unitary. We have that

\begin{lemma}
 Let $V$ be a $\ul{p}$-periodic function and $u \in \ell^2(\Z^d)$.
 Then
 \be
  \widehat{Vu}(\ul{x}) = \sum_{\ul{k}\in\mathbb{B}} \widehat{V}(\ul{k}) 
   \hat{u}(\ul{x} - \ul{k}).
 \ee
\end{lemma}

\begin{proof}
 This is a computation.
\end{proof}

The computations of this section seem to depend on the period.
However, the choices where made in a such a way, that they
are compatible. For example if one views $f:\Z^d\to\C$ as a $2 \ul{p}$
periodic function instead of a $\ul{p}$-periodic one, then 
the Fourier coefficients stay the same.

\subsection{Momentum representation of the periodic Schr\"odinger
 operator}

The discrete Laplacian $\Delta: \ell^2(\Z^d)\to\ell^2(\Z^d)$
is defined by $\Delta u(\ul{n}) = \sum_{j=1}^{d} (u(\ul{n}-b_j)+u(\ul{n}+b_j))$.
We recall that its Fourier transform is given by
\[
 F\Delta F^{-1} \hat{u}(\ul{x}) = 
 \left(\sum_{j=1}^{d} 2\cos(2\pi x_j)\right) \hat{u}(\ul{x}).
\]
In particular, if $V(\ul{n}) = \sum_{\ul{x}\in\mathbb{B}} \widehat{V}(\ul{k}) e(\ul{k} \cdot \ul{n})$ is a $\ul{p}$-periodic potential,
then the Schr\"odinger equation $(\Delta + V)u=Eu$ takes the form
\be
 \left(\sum_{j=1}^{d} 2\cos(2\pi x_j)\right) \hat{u}(\ul{x})
  + \sum_{\ul{k}\in\mathbb{B}} 
     \widehat{V}(\ul{k}) \hat{u}(\ul{x} - \ul{k}) 
  = E \hat{u}(\ul{x})
\ee
in Fourier variables.
It is easy to see that the equations involving
$\{\hat{u}(\ul{x} + \ul{k})\}_{\ul{k}\in\mathbb{B}}$
for
\be
 \ul{x}\in\mathbb{V} = \times_{j=1}^{d} [0, \frac{1}{p_j}),
  \quad \mathbb{B} = \times_{j=1}^{d} \left\{0,\frac{1}{p_j},\dots, \frac{p_j-1}{p_j}
   \right\}
\ee
are all independent of each other. Define the space
$L^2(\mathbb{V}\times\mathbb{B})$ as the set of all
maps $f:\mathbb{V}\times\mathbb{B} \to \C$ with norm
\be
 \|f\|_{L^2(\mathbb{V}\times\mathbb{B})}^2 = 
  \sum_{k\in\mathbb{B}} \int_{\mathbb{V}} |f(\ul{x}, \ul{k})|^2 d\ul{x}.
\ee
Introduce the map 
$W = \widetilde{W} F:\ell^2(\Z^d)\to L^2(\mathbb{V}\times\mathbb{B})$
where
\be\begin{split}
 \widetilde{W}: L^2(\T^d) &\to L^2(\mathbb{V} \times \mathbb{B}) \\
 (\widetilde{W} \hat{u})(\ul{x}, \ul{k}) &= 
   \frac{\hat{u}(\ul{x} + \ul{k})}{\left(\prod_{j=1}^{d} p_j\right)^\frac{1}{2}}.
\end{split}\ee

\begin{lemma}\label{lem:propW}
 The map $W: \ell^2(\Z^d) \to L^2(\mathbb{V}\times\mathbb{B})$
 is unitary. Furthermore, if the support of $u$
 \be
  \mathrm{supp}(u) = \{\ul{n}\in\Z^d:\quad u(\ul{n})\neq 0\}
 \ee
 is finite, then
 \be
  |Wu(\ul{x}, \ul{k})| \leq 
   \left(\frac{\#(\mathrm{supp}(u))}
   {\prod_{j=1}^{d} p_j}\right)^\frac{1}{2}
   \|u\|.
 \ee
\end{lemma}

\begin{proof}
 Since $F$ is unitary, it suffices to check that $\widetilde{W}$
 is unitary, but this follows directly. For the second claim,
 observe that
 \[
  Wu(\ul{x},\ul{k}) = \frac{1}{P^{1/2}} 
  \sum_{\ul{n}\in\mathrm{supp}(u)}
  e(-\ul{n}\cdot (\ul{x}+\ul{k}) u(\ul{n})
 \]
 which implies claim bu Cauchy--Schwarz.
\end{proof}

Also define for $\ul{\theta}\in\mathbb{V}$ the operator
$\widehat{H}_{\ul{p}, \ul{\theta}}:
\ell^2(\mathbb{B})\to\ell^2(\mathbb{B})$
\be
 \widehat{H}_{\ul{p}, \ul{\theta}} v(\ul{\ell}) = 
  \left(\sum_{j=1}^{d} 
   2\cos\left(2\pi (\ell_j + \theta_j)\right)\right) v(\ul{\ell}) + 
   \sum_{\ul{k}\in\mathbb{B}} \widehat{V}(\ul{k}) v(\ul{\ell} - \ul{k}).
\ee
We define the operator $\widehat{V}\ast:\ell^2(\mathbb{B})\to\ell^2(\mathbb{B})$ by
\be
 \widehat{V}\ast\psi(\ul{\ell}) = \sum_{\ul{k}\in\mathbb{B}} 
  \widehat{V}(\ul{k}) 
   v(\ul{\ell} - \ul{k}).
\ee
This way $\widehat{H}_{\ul{p}, \ul{\theta}} = \widehat{H}^{0}_{\ul{p},\ul{\theta}} + \widehat{V}\ast$. We define
\be\begin{split}
 \widehat{H}_{\ul{p}}: L^2(\mathbb{V}\times\mathbb{B}) 
  & \to L^2(\mathbb{V}\times\mathbb{B})\\
 (\widehat{H}_{\ul{p}} u)(\ul{x}, \ul{k})&= 
  (\widehat{H}_{\ul{p}, \ul{x}} u(\ul{x},.))(\ul{k}).
\end{split}\ee
The following proposition provides the first
form of {\em Floquet--Bloch decomposition} of
the periodic operator $H=\Delta+V$. We will encounter
a second one in Subsection~\ref{ss:spacebasis}.

\begin{proposition}
 We have that
 \be
  W H W^{-1} = \widehat{H}_{\ul{p}}.
 \ee
\end{proposition}

\begin{proof}
 This follows from the preceding computations.
\end{proof} 

We note that this gives a decomposition
of $H$ as a direct integral in the sense
of Section~XIII.16. of \cite{rs4}.

We furthermore wish to point out at this point
the following periodicity of $\widehat{H}_{\ul{p}, \ul{\theta}}$.
First observe that our definition of
$\widehat{H}_{\ul{p}, \ul{\theta}}$ makes
sense for any $\ul{\theta} \in \R^d$ and
even $\ul{\theta} \in \C^d$. We have that

\begin{lemma}\label{lem:hatHper}
 Let $1\leq j \leq d$. The operators
 \be
  \widehat{H}_{\ul{p}, \ul{\theta}},\quad
  \widehat{H}_{\ul{p}, \ul{\theta} + \frac{1}{p_j} b_j}
 \ee
 are unitairily equivalent.
\end{lemma}

\begin{proof}
 The unitary equivalence is given by $U$ defined
 by $U \psi( \ul{n} ) = \psi(\ul{n} + \frac{1}{p_j} b_j)$.
\end{proof}

\subsection{Analytic parametrization of eigenvalues and
 absolutely continuous spectrum}

To simplify the notation, we will now fix $\ul{\theta}^{\perp} = \{\theta_j\}_{j=2}^{d}$
and define
\be
 A(t) = \widehat{H}_{\ul{p}, (t, \ul{\theta}^{\perp})}.
\ee
Clearly $A(t)$ is an analytic family of operators
defined for every $t\in \C$.

\begin{proposition}\label{prop:lamjnonconst}
 The eigenvalues $\lambda_j(t)$ of $A(t)$ can be chosen to be analytic functions of $t$.
 Furthermore, each of these $\lambda_j(t)$ is a non-constant
 function of $t$.
\end{proposition}

\begin{proof}
 The first claim follows from $t \mapsto A(t)$ being an analytic map
 and for example Theorem II.6.1. in \cite{kato}.
 
 For the second claim observe that as $\im(t) \to \infty$, we have that
 $\|A(t)^{-1}\| \to 0$, since $A(t)$ is dominated by the diagonal
 containing values of size $\gtrsim \E^{t}$. Now
 \[
  \|A(t)^{-1}\| = \max_{j} \frac{1}{|\lambda_j(t)|}
 \]
 implies that $|\lambda_j(t)| \to \infty$ as $\im(t) \to \infty$.
 This implies that these are non-constant.
\end{proof}

The next theorem shows that the spectrum is
absolutely continuous. It will be used in the 
Section~\ref{sec:ids} to study the integrated
density of states and the spectral measure.
Furthermore, it provides essential information
on the nature of the spectrum.

\begin{theorem}\label{thm:ac}
 The spectrum of $H$ is purely absolutely continuous.
\end{theorem}

\begin{proof}
 See \cite{rs4} Theorem XIII.100.
\end{proof}

\subsection{Space basis}
\label{ss:spacebasis}

In this section, we derive a second type
of Floquet--Bloch decomposition.
Let $\ul{\theta}\in\mathbb{V}$ and define
the space $\ell^2_{\ul{p}, \ul{\theta}}(\Z^d)$ as the set
of all functions $u: \Z^d \to \C$ such that
\be
 u(\ul{n} + p_j b_j) = e(p_j \theta_j) u(\ul{n}),\quad j=1,\dots,n,\ \ul{n}\in\Z^d,
\ee
with the norm
\be
 \| u \|^2 = \sum_{1 \leq n_1 \leq p_1} \dots \sum_{1 \leq n_d \leq p_d} |u(\ul{n})|^2.
\ee 
Define
\be\begin{split}
 U_{\ul{p}, \ul{\theta}} :\ell^2(\mathbb{B})&\to\ell^2_{\ul{p},\ul{\theta}}(\Z^d),\\
 U_{\ul{p}, \ul{\theta}}\varphi(\ul{n})&=
  \sum_{\ul{k}\in\mathbb{B}} \varphi(\ul{k}) 
   e\left(\sum_{j=1}^{d} \left(\theta_j + k_j\right) \cdot n_j\right).
\end{split}\ee
One can check that this operator is unitary.
Denote by $H_{\ul{p},\ul{\theta}}$ the restriction of
$H$ to $\ell^2_{\ul{p},\ul{\theta}}(\Z^d)$. 

\begin{lemma}
 We have
 \be
  \widehat{H}_{\ul{p},\ul{\theta}} = 
   U_{\ul{p}, \ul{\theta}}^{\ast} H_{\ul{p},\ul{\theta}}
    U_{\ul{p},\ul{\theta}}
 \ee
 In particular $\widehat{H}_{\ul{p},\ul{\theta}}$
 and $H_{\ul{p},\ul{\theta}}$ have the same eigenvalues.
\end{lemma}

\begin{proof}
 This is a computation.
\end{proof}

Define $\mathbb{P} = \times_{j=1}^{d} \{1,\dots,p_j\}$,
$\Delta_{\ul{p}, \ul{\theta}}^j : \ell^2(\mathbb{P})
\to\ell^2(\mathbb{P})$ by
\be
 \Delta_{\ul{p}, \ul{\theta}}^j u(\ul{n}) =
  \begin{cases} 
   u(\ul{n} - b_j) + u(\ul{n} + b_j),&2 \leq n_j\leq p_j - 1\\
   u(\ul{n} - b_j) + e(\theta_j) u(\ul{n} + (p_j - 1) \cdot b_j),& n_j =  1\\
   u(\ul{n} - b_j) + e(-\theta_j) u(\ul{n} - (p_j - 1) \cdot b_j),& n_j = p_j,
 \end{cases}
\ee
and $\Delta_{\ul{p}, \ul{\theta}} = \sum_{j=1}^{d} \Delta_{\ul{p}, \ul{\theta}}^j$.
We note that $\ell^2_{\ul{p}, \ul{\theta}}(\Z^d)$
is isomorphic to $\ell^2(\mathbb{P})$ and that
$H \psi = E \psi$ if and only if
\be
 (\Delta_{\ul{p}, \ul{\theta}} + V)\ti{\psi} = E \ti{\psi}
\ee
where $\ti{\psi}$ denotes the restriction of $\psi$
to $\mathbb{P}$.

For $\ul{d}, \ul{p} \in (\Z_+)^d$, we define
\be
 (\ul{d}\ast\ul{p})_j = d_j \cdot p_j.
\ee
We have that

\begin{lemma}\label{lem:l2theta}
 Let $\ul{d}, \ul{p} \in (\Z_+)^d$,
 $\ul{\theta} \in \mathbb{V}_{\ul{d}\ast\ul{p}}$.
 Then
 \be
  \ell^2_{\ul{d}\ast\ul{p},\ul{\theta}}(\Z^d) =
   \bigoplus_{\varphi \in \Phi}
    \ell^2_{\ul{p},\ul{\varphi}}(\Z^d)    
 \ee
 where
 \be
  \Phi = \left\{(\theta_j + \frac{\ell_j}{p_j d_j})_{j=1}^{d},
   \quad \ul{\ell} \in \times_{j=1}^{d} \{0, \dots, d_j-1\}\right\}.
 \ee
\end{lemma}

\begin{proof}
 If $u\in \ell^2_{\ul{p},\ul{\varphi}}(\Z^d)$, then
 \[
  u(\ul{n} + d_j p_j b_j ) = e(p_j d_j \varphi_j) u(\ul{n})
 \]
 which is equal to $e(p_jd_j\theta_j) u(\ul{n})$
 if and only if $\ul{\varphi}\in \Phi$.
 Hence, we have that
 \[
  \ell^2_{\ul{d}\ast\ul{p},\ul{\theta}}(\Z^d) \supseteq
   \bigoplus_{\varphi \in \Phi}
    \ell^2_{\ul{p},\ul{\varphi}}(\Z^d).    
 \]
 A counting argument shows that the dimensions
 of the spaces agree. Hence, equality holds.
\end{proof}

This lemma implies that the eigenvalues
of $H_{\ul{d}\ast\ul{p}, \ul{\theta}}$
are the union over the eigenvalues
of $H_{\ul{p}, \ul{\varphi}}$ with
$\ul{\varphi}\in\Phi$.

%%%%%%%%%%%%%%%%%%%%%%%%%%%%%%%%%%%%%%%%%%%%%%%%5
%
%
%

\section{Bands and Gaps}
\label{sec:bands}

The goal of this section is to introduce the
language of bands and gaps and to prove basic
results about them.

\begin{lemma}\label{lem:evorder}
 The operator $\widehat{H}_{\ul{p}, \ul{\theta}}$ has 
 $P=\prod_{j=1}^{P} p_j$
 eigenvalues. Orders these
 \be
  E_1(\ul{\theta}) \leq E_2(\ul{\theta}) \leq \dots \leq E_P(\ul{\theta})
 \ee
 in increasing order. Then
 \be\label{eq:minmax}
  E_j(\ul{\theta}) = \min_{\dim(V) = j-1}
   \max_{\stackrel{\psi \perp V}{\|\psi\|=1}}
    \spr{\psi}{ \widehat{H}_{\ul{p}, \ul{\theta}} \psi}
 \ee
 in particular the $E_j:\mathbb{V}\to\R$ are continuous functions.
\end{lemma}

\begin{proof}
 The number of eigenvalues is $P$, since $\ell^2(\mathbb{B})$ is $P$-dimensional. \eqref{eq:minmax} is the min-max principle
 and implies that the $E_j$ are continuous.l
\end{proof}

We will now relate the properties of the functions $E_j(\ul{\theta})$
to the spectrum of $H$. Define
\be
 E_j^- = \min_{\ul{\theta}\in\mathbb{V}} E_j(\ul{\theta}),\quad
 E_j^+ = \max_{\ul{\theta}\in\mathbb{V}} E_j(\ul{\theta}).
\ee

\begin{definition}
 The intervals $[E_j^-, E_j^+]$ are called bands.
 If $E_j^+ < E_{j+1}^-$, then $(E_j^+, E_{j+1}^-)$
 is called a gap.
\end{definition}

It is clear that the bands are subset of the spectrum $\sigma(H)$
of $H$ and that gaps of the resolvent. We furthermore
note

\begin{theorem}\label{thm:numgapsfinite}
 The spectrum of the $\ul{p}$-periodic Schr\"odinger operator
 $H = \Delta + V$ is given by
 \be
  \sigma(\Delta + V) = \bigcup_{j=1}^{P} [E_{j}^-, E_{j}^+].
 \ee 
 In particular, it contains at most $P - 1 = \prod_{j=1}^{d} p_j - 1$ 
 many gaps.
\end{theorem}

\begin{proof}
 See Theorem~XIII.85 in \cite{rs4} and use that the
 functions $E_j$ are continuous.
\end{proof}

One should point out that all these gaps can occur
as the following is a somewhat degenerate example shows.
Let $\ul{m}^{\ell}$ be an enumeration of 
$\mathbb{P} = \times_{j=1}^{d} \{1, \dots, p_j\}$
and define a potential $V$ by
\be
 V(\ul{n}) = (4 d + 1) \ell
\ee
whenever $n_j = m^{\ell}_j \pmod{p_j}$ for $j=1,\dots,d$.
It is then relatively easy to check that
\[
 \sigma(\Delta + V) \cap ((4d + 1)\ell + 2d, (2d + 1)\ell + 2d + 1) = \emptyset
\]
for $\ell = 1, \dots, P - 1$
and
\[
 \sigma(\Delta + V) \cap ((4d +1) \ell - 2d, (4 d + 1) \ell + 2d) \neq \emptyset
\]
for $\ell = 1, \dots, P$. Hence $\sigma(\Delta + V)$ contains
at least $P - 1$ many gaps, and by Theorem~\ref{thm:numgapsfinite}
exactly $P - 1$.

\subsection{Non-constancy}

We will now discuss further properties of the eigenvalue
parametrization from Lemma~\ref{lem:evorder}. 
For $X$ a topological space, we will call
a point $x \in X$ a {\em point of increase} of 
a function $f:X\to\R$
if for any open set $x \in U$, we have
\be
 \inf_{y\in U} f(y) < f(x) < \sup_{y\in U} f(y).
\ee
Or in words, we can find $y, \ti{y}$ arbitrarily close to $x$
such that $f(y) < f(x) < f(\ti{y})$. The main result is

\begin{theorem}\label{thm:ptofincrease}
 Let $[E_j^-, E_j^+]$ be a band. Then for any $E \in (E_j^-, E_j^+)$,
 there exist infinitely many points $\ul{\theta}$ such that $E_j(\ul{\theta}) = E$
 and $\ul{\theta}$ is a point of increase.
\end{theorem}

A refinement of the following argument, noting that
$E_j^{-1}(E)$ is the boundary of an open set, shows
that the set of $\ul{\theta}$ described in this
theorem has Hausdorff dimension at least $d-1$.
For the proof, we will need the following lemma

\begin{lemma}\label{lem:evjnonconst}
 The eigenvalues $E_j(\ul{\theta})$ defined in Lemma~\ref{lem:evorder}
 are not constant on an open set.
\end{lemma}

\begin{proof}
 Assume that $E_j$ was constant on an open set. This
 would in imply that $\lambda_\ell(t)$ as in Proposition~\ref{prop:lamjnonconst} was constant for some $\ell$.
 A contradiction.
\end{proof}

\begin{proof}[Proof of Theorem~\ref{thm:ptofincrease}]
 The two sets $A = E_j^{-1} ((-\infty, E))$ and $B = E_j^{-1}((E,\infty))$
 are disjoint and open in $\mathbb{V}$. Hence, 
 \[
  C = [0,1]^d \setminus (A\cup B) = E_j^{-1}(\{E\})
 \]
 must be infinite. By the previous lemma
 $C$ does not contain an open set.
 This implies that also
 \[
  \widetilde{C} = C \cap \ol{A} \cap \ol{B}
 \]
 is infinite, and by definition contains points of increase
 of $E_j$. Hence, we are done.
\end{proof}

We furthermore note the obvious lemma

\begin{lemma}\label{lem:incimpint}
 Let $\ul{\theta}$ be a point of increase of $E_j$,
 then $E_j(\ul{\theta})$ is in the interior of a band.
\end{lemma}

\subsection{Stability of the spectrum being an interval}

We start with

\begin{definition}
 Let $H$ be a $\ul{p}$-periodic Schr\"odinger operator and
 $\delta \in \R$. We say that the bands of $H$ are $\delta$-overlapping
 if
 \be
  E_{j+1}^- - E_{j}^+ \geq \delta
 \ee
 for $j=1,\dots, P -1$. The bands of $H$ are called overlapping
 if they are $\delta$-overlapping for some $\delta > 0$.
\end{definition}

In particular, if the bands of $H$ are overlapping, then
the spectrum of $H$ is an interval. We allow for negative
values of $\delta$ so statements like the next theorem
become possible without restrictions on $\|V\|_{\infty}$.

\begin{theorem}\label{thm:bandoverlap}
 Let the bands of $H$ be $\delta$-overlapping.
 Then the bands of $H + V$ are $\delta- 2\|V\|_{\infty}$-overlapping.
\end{theorem}

For $\ul{p}$-periodic $V:\Z^d\to\R$,
we denote by $E_j(\ul{\theta}, V)$ the eigenvalues of
$H_{\ul{p}, \ul{\theta}} - V$ as defined in Lemma~\ref{lem:evorder}.
We note the following simple lemma.

\begin{lemma}
 We have
 \be
  E_j(\ul{\theta}, 0) - \|V\|_{\infty} \leq E_j(\ul{\theta}, V)
  \leq E_j(\ul{\theta}, 0) + \|V\|_{\infty}.
 \ee
 In particular, if
 \be
  \|V\|_{\infty} \leq \frac{1}{2}
   \min_{1\leq j\leq P-1} (E_j^{+} - E_{j+1}^{-})
 \ee
 then the spectrum of $H + V$ is an interval.
\end{lemma}

\begin{proof}
 The first claim follows from the min-max principle
 \eqref{eq:minmax}. The second
 one from the first, since it implies
 $E_j^{+} - E_{j+1}^{-} \geq 0$
 which exactly says that $\sigma(H+V)$ is an interval.
\end{proof}

\begin{proof}[Proof of Theorem~\ref{thm:bandoverlap}]
 This follows from the previous lemma.
\end{proof}

The following lemma will be used in our construction
of limit-periodic potentials.

\begin{lemma}\label{lem:uniformoverlap}
 Let $H$ be a $\ul{p}$-periodic Schr\"odinger operator
 and $\mathcal{V}$ a compact set of $\ul{p}$-periodic
 potentials. Assume that for every $V\in\mathcal{V}$, the bands of
 $H + V$ are overlapping. Then there exists $\delta > 0$
 such that for all such $V$, the bands of $H + V$
 are $\delta$-overlapping.
\end{lemma}

\begin{proof}
 It is easy to see that the map $V \mapsto E_j(\ul{\theta}, V)$
 is continuous. Hence, also the map
 \[
  g: V \mapsto \inf_{j} \left(E_{j+1}^-(V) - E_j^+(V)\right).
 \]
 is continuous. Since $\mathcal{V}$ is compact
 and $g(V)  > 0$, it follows that $\inf_{V\in\mathcal{V}} g(V) > 0$
 which is the claim.
\end{proof}

\subsection{An upper bound on the length of bands}

We have that

\begin{theorem}\label{thm:bandsize}
 Let $H$ be $\ul{p}$-periodic, then the length of bands
 is bounded by
 \be
  E_j^+ - E_j^- \leq 4 \pi \sum_{j=1}^{d} \frac{1}{p_j}.
 \ee
\end{theorem}

\begin{proof}
 Let $E_j^{\pm} = E_j(\ul{\theta}_\pm)$.
 Since $\ul{\theta}_{\pm} \in \mathbb{V}$, we have
 $|(\ul{\theta}_- - \ul{\theta}_+)_j| \leq \frac{1}{p_j}$.
 Define the family of operators
 \[
  A(t) = \widehat{H}_{\ul{p}, \ul{\theta}_- + t(\ul{\theta}_+ - \ul {\theta}_-)},
  \quad t\in [0,1]
 \]
 which is clearly analytic. Denote by
 $\lambda_{\ell}(t)$ an analytic parametrization
 of the eigenvalues of $A(t)$. We have that
 $\lambda_{\ell}'(t) = \spr{\psi}{A'(t) \psi}$,
 where $\psi$ is any normalized solution of
 $A(t) \psi = \lambda_{\ell}(t) \psi$.
 Hence, we have that $|\lambda_{\ell}'(t)|\leq\|A'(t)\|$
 and since $E_j'(t) = \lambda_{\ell}'(t)$
 for some $\ell = \ell(t)$ except at finitely many
 points, we obtain that
 \[
  |E_j^{+} - E_{j}^-|  \leq \sup_{t\in [0,1]} \|A'(t)\|.
 \]
 One can easily compute that
 $\|\partial_{\theta_j} \widehat{H}_{\ul{p},\ul{\theta}}\| \leq 4\pi$.
 Hence, we obtain that
 \[
  A'(t) = \sum_{j=1}^{d} (\ul{\theta}_+-\ul{\theta}_-)_j 
   \cdot \partial_{\theta_j} \widehat{H}_{\ul{p},\ul{\theta}}
 \]
 satisfies $\|A'(t)\| \leq 4\pi\sum_{j=1}^{d} \frac{1}{p_j}$,
 which is the claim.
\end{proof}

%%%%%%%%%%%%%%%%%%%%%%%%%%%%%%%%%%%%%%%%%%%%%%%%5
%
%
%

\section{The integrated density of states and spectral measures}
\label{sec:ids}

The goal of this section is to investigate two quantities
related to the spectrum of $H$: the integrated density of
states and the spectral measures.

\subsection{The integrated density of states}

Let $\ell \geq 1$ and denote by $\Lambda_\ell$, the 
rectangle
\be
 \Lambda_{\ell} = \times_{j=1}^{d} \{1, \dots, \ell \cdot p_j\}
\ee
and by $\#\Lambda_\ell$ the number of elements in $\Lambda_\ell$.
We denote by $H^{\Lambda_{\ell}}$ the restriction of $H$
to $\ell^2(\Lambda_{\ell})$ and by 
\be
 \tr\left(P_{(-\infty, E)}\left(H^{\Lambda_{\ell}}\right)\right)
\ee
the number of eigenvalues of $H^{\Lambda_{\ell}}$ less than $E$.

\begin{theorem}\label{thm:ids1}
 The limit
 \be\label{eq:defids}
  k(E) = \lim_{\ell\to\infty} \frac{1}{\#\Lambda_{\ell}}
  \tr\left(P_{(-\infty, E)}\left(H^{\Lambda_{\ell}}\right)\right)
 \ee
 exists.
 Furthermore with $E_j(\ul{\theta})$ as in Lemma~\ref{lem:evorder},
 we have
 \be
  k(E) = \frac{1}{P} \int_{\mathbb{V}}
  \#\{j:\quad E_j(\ul{\theta}) \leq E\} d\ul{\theta}.
 \ee
\end{theorem}

By Lemma~\ref{lem:hatHper}, we have that
\be
 \tr\left(P_{(-\infty, E)}\left(H^{\Lambda_{\ell}}\right)\right)
 =\tr\left(P_{(-\infty, E)}\left(H^{\Lambda_{\ell}+\ul{n}\ast\ul{p}}\right)\right)
\ee
for any $\ul{n}\in\Z^d$. Hence, the limit in Theorem~\ref{thm:ids1}
is somewhat more general than we claim here.
We first prove 

\begin{lemma}
 Let $\ul{\theta}\in\mathbb{V}$. Then
 \be
  \lim_{\ell\to\infty} \frac{1}{\#\Lambda_{\ell}}
   \left(
    \tr\left(P_{(-\infty, E)}\left(H^{\Lambda_{\ell}}\right)\right)-
     \tr\left(P_{(-\infty, E)}\left(H_{\ell\cdot\ul{p}, \ul{\theta}}\right)\right)
      \right) = 0.
 \ee
\end{lemma}

\begin{proof}
 $H^{\Lambda_{\ell}}$  and $H_{\ell \cdot \ul{p}, \ul{\theta}}$ differ
 by a rank 
 \[
  r_{\ell} = 2 \ell^{d-1} \sum_{j=1}^{d} \prod_{k \neq j} p_j
 \]
 perturbation. Hence,
 \[
  \frac{1}{\#\Lambda_{\ell}} \left(
   \tr\left(P_{(-\infty, E)}\left(H^{\Lambda_{\ell}}\right)\right)-
    \tr\left(P_{(-\infty, E)}\left(H_{\ell\cdot\ul{p}, \ul{\theta}}\right)\right)
     \right) = O(\frac{1}{\ell})
 \]
 and the claim follows.
\end{proof}

By Lemma~\ref{lem:l2theta}, we have that
\be
 \tr\left(P_{(-\infty, E)}\left(H_{\ell\cdot\ul{p}, \ul{\theta}}\right)\right)
  = \sum_{\ell \ul{\varphi} = \ul{\theta} \pmod{1}}
   \tr\left(P_{(-\infty, E)}\left(H_{\ul{p}, \ul{\varphi}}\right)\right).  
\ee

\begin{proof}[Proof of Theorem~\ref{thm:ids1}]
 By the previous two results, it suffices to show that
 \[
  \lim_{\ell\to\infty} \frac{1}{l^d} \sum_{\ell \ul{\varphi} = \ul{\theta} \pmod{1}}
   \tr\left(P_{(-\infty, E)}\left(H_{\ul{p}, \ul{\varphi}}\right)\right)
  = \int_{\mathbb{V}} 
     \#\{j:\quad E_j(\ul{\theta}) \leq E\} d\ul{\theta}.
 \] 
 For this, observe that
 \[
  \tr\left(P_{(-\infty, E)}\left(H_{\ul{p}, \ul{\varphi}}\right)\right) 
   = \#\{j:\quad E_j(\ul{\varphi}) < E\}.
 \]
 Furthermore, by Theorem~XIII.83.(e) in \cite{rs4} and
 Theorem~\ref{thm:ac} the measure of
 $\ul{\varphi} \in \mathbb{V}$ such that
 $E_j(\ul{\varphi}) = E$ for some $j$, is $0$. Hence,
 the result follows by a convergence theorem for integrals.
\end{proof}

The results of Craig--Simon \cite{cs1} imply that
the integrated density of states is log H\"older
continuous. Furthermore, it follows from
Theorem~\ref{thm:ac} that it is absolutely
continuous. It would be interesting to obtain
further reguarity results, see the discussion
in the next subsection.

The definition of the limit in \eqref{eq:defids}
is not as general as possible. One can show that
if $\Lambda_{t}$ is a Folner sequence for $\Z^d$,
then
\be
 k(E) = \lim_{t\to\infty} \frac{1}{\#\Lambda_t}
 \tr\left(P_{(-\infty,E)}(H^{\Lambda_t})\right).
\ee

The function $k$ defined in Theorem~\ref{thm:ids1}
is clearly increasing. Hence, there exists a measure $\nu$
such that $k(E) = \nu((-\infty,E))$. This measure
is called {\em density of states}.

\begin{lemma}\label{lem:ids2}
 For $\im(z) > 0$, we have
 \be
  \int \frac{d\nu(t)}{t - z} = \frac{1}{P} \sum_{j=1}^{P}
   \int_{\mathbb{V}} \frac{d\ul{\theta}}{E_j(\ul{\theta}) - z}
 \ee
 where $P=\prod_{j=1}^{d} p_j$.
\end{lemma}

\begin{proof}
 We have that $\int \frac{d\nu(t)}{t - z} = - \int \frac{k(t) dt}{(t -z)^2}$ and that
 \[
  k(E) = \frac{1}{P} \sum_{j=1}^{P} \int_{\mathbb{V}}
   \chi_{(-\infty, E_j(\ul{\theta}))}(E) d\ul{\theta}.
 \]
 The claim then follows by Fubini and a quick computation.
\end{proof}

\subsection{Spectral measures}

By Theorem~\ref{thm:ac} all the spectral measures
are absolutely continuous. The goal of this section
will be to give more quantitative information.
Given $u \in \ell^2(\Z^d)$, we denote by $\mu^{u}$ the measure
satisfying
\be
 \spr{u}{(H - z)^{-1} u} = \int \frac{d\mu^{u}(t)}{t - z}
\ee
for $\im(z) > 0$. 
The main result of this section is

\begin{theorem}\label{thm:specmeasac}
 Let $H$ be $\ul{p}$-periodic and $u \in \ell^2(\Z^d)$
 with finite support.
 Then the spectral measure $\mu^u$ is absolutely continuous
 with respect to the density of states $\nu$.

 Furthermore there exists $C = C(\ul{p}, u) > 0$, such that
 \be
  \left\|\frac{d\mu^{u}}{d\nu}\right\|_{L^{\infty}(\R)}
   \leq C.
 \ee
\end{theorem}

Here $\frac{d\mu^{u}}{d\nu}$ denotes the Radon-Nikodym
derivative of $\mu^u$ with respect to $\nu$.
At this point, I would like to ask

\begin{question}\label{ques:idslp}
 Let $B > 0$.
 Do there exist $q > 1$ and $C > 0$
 such that we have for all $\ul{p}$-periodic
 $V$ with $\|V\|_{\infty} \leq B$ that
 \be
  \left\|\frac{d\nu}{dE}\right\|_{L^q(\R)} \leq C?
 \ee
\end{question}

This is true in dimension one, see \cite{dg2}.
A positive answer to this question
would allow us to obtain an uniform $L^q$ bound
on all spectral measures. This result would then in
turn allow us to carry over the construction of limit-periodic
potentials with absolutely continuous spectrum of
Avron--Simon \cite{as1} (see also \cite{dg2})
to the multi-dimensional case.

\bigskip

We now begin with the proof of Theorem~\ref{thm:specmeasac}.
Denote $u_{\ul{\theta}}(\ul{k}) = (Wu)(\ul{\theta}, k)$
and by $E_j(\ul{\theta})$, $\psi_j(\ul{\theta})$ the eigenvalues
and eigenfunctions of $\widehat{H}_{\ul{p}, \ul{\theta}}$. We clearly
have that
\be
 \int\frac{d\mu^{u}(t)}{t-z} = \sum_{j=1}^{P}
  \int_{\mathbb{V}} |\spr{\psi_j(\ul{\theta})}{u_{\ul{\theta}}}|^2
   \frac{d \ul{\theta}}{E_j(\ul{\theta}) - z}.
\ee
On the other hand, we have seen in the previous
part of this section that
\be
 \int\frac{d\nu(t)}{t - z} = \frac{1}{\prod_{j=1}^{d} p_j}
  \sum_{j=1}^{P} 
   \int_{\mathbb{V}} \frac{d\ul{\theta}}{E_j(\ul{\theta}) - z}.
\ee
We are now ready for

\begin{proof}[Proof of Theorem~\ref{thm:specmeasac}]
 The densities of a measure are given by
 \[
  \frac{d\mu^{u}(E)}{dE} = 
   \lim_{\eps\to 0} \frac{1}{\pi} 
    \im\left(\spr{u}{(H - (E+\I\eps))^{-1} u}\right).
 \]
 Hence, we conclude from the previous two formulas
 that
 \[\begin{split}
  \frac{d\mu^{u}}{d E} &\leq \frac{1}{\pi}
   \sup_{\stackrel{\ul{\theta}\in [0,1]^d}{1\leq j\leq P}} 
   |\spr{\psi_j(\ul{\theta})}{u_{\ul{\theta}}}|^2
     \cdot \lim_{\eps\to 0} \int_{\mathbb{V}} \sum_{j=1}^{P}
      \frac{\eps d\ul{\theta}}{(E_j(\ul{\theta}) - E)^2 + \eps^2} 
       \\ & \leq C(u, \ul{p}) \frac{d\nu}{dE},    
 \end{split}\]
 where we used $|\spr{\psi_j(\ul{\theta})}{u_{\ul{\theta}}}|$ is uniformly
 bounded, by Lemma~\ref{lem:propW} and $\psi_j$ being normalized.
\end{proof}

%%%%%%%%%%%%%%%%%%%%%%%%%%%%%%%%%%%%%%%%%%%%%%%%%%%%%%%%%%%%%%%%%%%%%
%
%
%

\section{Simplicity of the spectrum for coprime periods}
\label{sec:specsimple}

In this section, we restrict ourself to dimension two.
For this reason, we will denote the periods of the
potential $V$ by $(p,q)$ and the angles in the 
Floquet--Bloch decomposition by $\theta \in [0, \frac{1}{p})$,
$\varphi \in [0,\frac{1}{q})$.
We furthermore recall that $p,q \geq 2$ are called coprime,
if they have no common divisor.

The following theorem will be proven in this section.
It is a technical result that will allow the constructions
in the following sections.

\begin{theorem}\label{thm:coprimeevsimple}
 Let $p, q\geq 2$ be co-prime, $E\in\R$, and $V:\Z^2 \to\R$
 be $(p,q)$-periodic. Then the set
 \be
  \left\{(\theta,\varphi):\quad E\text{ is an eigenvalue of multiplicity
   $\geq 2$ of } \widehat{H}_{(p,q), (\theta,\varphi)}\right\}
 \ee
 is finite.
\end{theorem}

The analog statement in dimensions $d \geq 2$, is that
the set of $\ul{\theta}$ such that $E$ is an eigenvalue
of multiplicity at least two, has dimension less than
$d - 2$. I expect that proving this result and then
using it would be somewhat more involved, than what
is done here.

The proof of this theorem has essentially two parts. First
an algebraic reduction is performed allowing us to prove
the claim by proving a statement when $\im(\theta)$ or
$\im(\varphi)$ is large. In this regime the operator
is essentially diagonal, and the analysis of this takes
up the second step.

\subsection{Algebraic preparations}

Define
\be
 \widetilde{H}_0(u,v) \psi(k,l) = \left(
    e\left(\frac{k}{p}\right) u + e\left(-\frac{k}{p}\right) \frac{1}{u}
   +e\left(\frac{l}{q}\right) v + e\left(-\frac{l}{q}\right) \frac{1}{v}
   \right) \psi(k,l)
\ee
acting on $\ell^2([1,p] \times [1,q])$ and
$\widetilde{H}(u,v) = \widetilde{H}_0(u,v) + \widehat{V}\ast$.
We have that $\widehat{H}(\theta,\varphi) = \widetilde{H}(e(\theta), e(\varphi))$ if one identifies $\ell^2([1,p]\times [1,q])$
with $\ell^2(\mathbb{B})$ in the obvious way.

\begin{proposition}\label{prop:metabezout}
 Let $E \in \R$. Assume there exist $u_0, v_0 \in \C$ such that
 for no $u,v \in \C$, $E$ is an eigenvalue of multiplicity at least
 two of $\widetilde{H}(u_0,v)$ or $\widetilde{H}(u,v_0)$.
 Then the set
 \be\label{eq:metabezout}
  \left\{(\theta, \varphi):\quad E\text{ is an eigenvalue of multiplicity
   at least two of }\widehat{H}(\theta,\varphi)\right\}
 \ee
 is finite.
\end{proposition}

The proof of this result is based on ideas from algebraic
geometry in particular B\'ezout's theorem and resultants.
We recall main properties of resultants for the convenience of the reader.
For further details see Chapter~3 of \cite{kirwan}.
Given two polynomials $f(x) = \sum_{j=0}^{d} a_j x^j$ and
$g(x) = \sum_{\ell=0}^{D} b_{\ell} x^{\ell}$, their
resultant $R(f,g)$ is defined as the determinant
of the {\em Sylvester matrix}
\be
 \begin{pmatrix}
 a_0 & a_1 & \dots & a_d    &   0 & 0 & \dots & 0 \\
 0   & a_0 & a_1   & \dots  & a_d & 0 & \dots & 0 \\
 \vdots & &        & \ddots &     &   &       & \vdots \\
 0   & 0   & \dots & 0      & a_1 & a_2 & \dots & a_d \\
 b_0 & b_1 & \dots & \dots  & b_D & 0 & \dots & 0 \\
 \vdots & &        & \ddots &     &   &       & \vdots \\
 0   & \dots &     & b_0    & \dots &  &      & b_D
 \end{pmatrix}. 
\ee
We have that $R(f,g) = 0$ if $f$ and $g$ have a common zero.

We now return to the original problem.
Define for $E\in\R$ the polynomial
\be
 P_E(u,v) = (uv)^{p\cdot q} \det(\widetilde{H}(u,v) - E).
\ee
We can write $P_E(u,v) = \sum_{j=0}^{2pq} a_j(u) v^j$
and $\partial_E P_E(u,v) = \sum_{j=0}^{2pq-1} b_j(u) v^j$
for some polynomials $a_j, b_j$ in $u$. Then we can
define the resultant of these two polynomials, which
will be a function of $u$
\be
 f(u) = R(P_E(u, .), \partial_E P_E(u, .)).
\ee
Similarly, we can define
\be
 g(v) = R(P_E(., v), \partial_E P_E(., v)).
\ee
Since $f$ and $g$ are polynomials, they are either
constant equal to $0$ or have finitely many zeros.

\begin{lemma}
 If $f$ and $g$ are not the constant zero function,
 then the number of points $(u,v)$ such that $E$
 is an eigenvalue of $\widetilde{H}(u,v)$ of multiplicity
 at least two, is finite.
\end{lemma}

\begin{proof}
 If $E$ is an eigenvalue of $\widetilde{H}(u,v)$ of multiplicity
 at least two, then we have that
 \[
  P_{E}(u, v) = \partial_E P_{E}(u,v) = 0.
 \]
 In particular, we have that $f(u) = g(v) = 0$. Hence,
 the set of $(u,v)$ where $E$ is an eigenvalue of multiplicity
 at least two is contained in
 \[
  \{(u,v):\quad f(u) = g(v) = 0\}
 \]
 which is finite.
\end{proof}

\begin{proof}[Proof of Proposition~\ref{prop:metabezout}]
 Our assumptions imply that $f(u_0) \neq 0$ and $g(v_0) \neq 0$.
 Hence, we are done.
\end{proof}

\begin{remark}
 The degree of $P_E(u,v)$ is $2 p \cdot q$ and
 of $\partial_E P_{E}(u,v)$ is $2 p \cdot q - 1$.
 Using this and an inspection of the previous
 argument shows that the set in \eqref{eq:metabezout}
 contains at most $(2pq)^2$ many points.
\end{remark}

\subsection{Perturbative analysis}
Now, we will verify the conditions of
Proposition~\ref{prop:metabezout}. We first observe
that the claim is symmetric in $u$ and $v$, so it
suffices to exhibit $u_0$. The proof of the existence
of $v_0$ is then similar.

Next, we will split the operator $(uv)^{p\cdot q} (\widetilde{H}(u,v)  - E)$
into diagonal and off-diagonal part. For this define
\begin{align}
 D(u,v)\psi(k,\ell) & = d_{k,\ell}(u,v) \psi(k,\ell) \\
  d_{k,\ell}(u,v) & = u^2 v e\left(\frac{k}{p}\right) + v e\left(-\frac{k}{p}\right)
  + u v^2 e\left(\frac{\ell}{q}\right) + u e\left(-\frac{\ell}{q}\right)
  - E u v,\\
 T(u,v)\psi(k,\ell) &= u v \sum_{\ti{k}, \ti{\ell}} 
  \widehat{V}(\ti{k},\ti{\ell}) \psi(k - \ti{k}, \ell - \ti{\ell}).
\end{align}
Then, we have that
\be
 (uv)^{p\cdot q} (\widetilde{H}(u,v)  - E) = D(u,v) + T(u,v).
\ee
A simple counting argument shows that
for every $u \in \C$,
$p(v) = P_{E}(u, v) = \det(D(u,v) + T(u,v))$
is a polynomial of degree $2 p \cdot q$. Hence, we have

\begin{lemma}
 If there exists $u$ such that there
 exist $2 p \cdot q$ different $v_j$ such that
 \be
  0 \in \sigma(D(u, v_j) + T(u, v_j))
 \ee
 is a simple eigenvalue, then the assumptions of Proposition~\ref{prop:metabezout}
 hold.
\end{lemma}

We will show the claim for all sufficiently small $u$.
We begin with the analysis of $D(u, v)$.
The next lemma is a simple computation.

\begin{lemma}
 Let $u$ be small, $k, \ell$ be given. The solutions of
 \be
  d_{k,\ell}(u, v) = 0
 \ee
 are given by
 \be
  v_+^{k,\ell} = -\frac{1}{u} e(-\frac{kq+\ell p}{pq}) + O(1),\quad
  v_-^{k,\ell} = u e(\frac{kq - \ell p}{pq}) + O(u^2).
 \ee
\end{lemma}

\begin{proof}
 Let
 \[
  A=u e\left(\frac{kq-\ell p}{pq}\right) - E e(\frac{\ell}{q})
   + \frac{1}{u} e(-\frac{kq+\ell p}{pq}),\quad
  B=e\left(-\frac{2\ell}{q}\right)
 \]
 such that $v_{\pm}^{k,\ell}$ satisfy
 the quadratic equation $v^2 + A v + B = 0$.
 Since $u$ is small, we obtain for the roots
 \[
  v_{\pm}^{k,\ell} = - \frac{1}{2} A (1 \pm 1 \mp \frac{2 B}{A^2}
   + O(\frac{B^2}{A^4}))
 \]
 which yields the claim after some computations.
\end{proof}

The next lemma can be proven by a computation.

\begin{lemma}
 There exists a constant $C > 0$.
 For $(k,\ell) \neq (\ti{k}, \ti{\ell})$
 \be
  d_{\ti{k},\ti{\ell}}(u, v_{+}^{k,\ell}) \geq \frac{C}{u},\quad
  d_{\ti{k},\ti{\ell}}(u, v_{-}^{k,\ell}) \geq C u
 \ee
 and
 \be
  |v_+^{k,\ell} - v_+^{\ti{k},\ti{\ell}}| \geq \frac{C}{u},\quad
  |v_+^{k,\ell} - v_-^{\ti{k},\ti{\ell}}| \geq \frac{C}{u},\quad
  |v_-^{k,\ell} - v_-^{\ti{k},\ti{\ell}}| \geq Cu.
 \ee
\end{lemma}

In the following, we will prove

\begin{proposition}\label{prop:asymvpm}
 For $u$ small enough, the $2pq$ zeros of $v\mapsto \det(D(u,v) + T(u,v))$
 are given by $\{\ti{v}_+^{k,\ell}, \ti{v}_-^{k,\ell}\}_{k,\ell}$
 with
 \be
  \ti{v}_+^{k,\ell} =v_+^{k,\ell} +O(1), 
   \quad
    \ti{v}_-^{k,\ell} = v_-^{k,\ell} + O(u^2)
 \ee
 Furthermore
 $0$ is a simple eigenvalue of $D(u,\ti{v}_+^{k,\ell}) + T(u,\ti{v}_+^{k,\ell})$
 and $D(u,\ti{v}_-^{k,\ell}) + T(u,\ti{v}_-^{k,\ell})$.
\end{proposition}

As discussed before, this finishes the proof of Theorem~\ref{thm:coprimeevsimple}.
The proof of this proposition will be given by a perturbative
analysis. We will first need

\begin{lemma}
 $D(u,v)+T(u,v)$ is normal.
\end{lemma}

\begin{proof}
 We have to show that $A(u,v)^* A(u,v) = A(u,v) A(u,v)^*$
 with $A(u,v) = D(u,v) + T(u,v)$.
 Since multiplying an operator by a scalar doesn't change this
 condition, it suffices to check that $\widetilde{H}(u,v)$ is normal. So 
 we have to show
 \[
  \widetilde{H}(u,v)^* \widetilde{H}(u,v) - \widetilde{H}(u,v) \widetilde{H}(u,v)^* = 0
 \] 
 for all $(u,v)\in\C^2$. 
 For $(u,v) \in \R^2$, $\widetilde{H}(u,v)$ is self-adjoint
 and thus the previous equation holds. By 
 analyticity of $(u,v) \mapsto \widetilde{H}(u,v)$
 the equation holds for all $(u,v)\in \C^2$ and
 we are done.
\end{proof}

We need the following general fact about normal matrices.
It is inspired by Section~9 of \cite{k1}.
We denote by $\sigma(A)$ the spectrum of $A$.

\begin{proposition}
 Let $A$ and $B$ be normal matrices, $\eps > 0$, and $t \in (0, \frac{1}{100})$.
 Assume that
 \begin{enumerate}
  \item $0$ is a simple eigenvalue of $A$.
  \item $\sigma(A) \cap \{z:\quad |z|< \eps\} = \{0\}$.
  \item $\|A - B\| \leq t \eps$.
 \end{enumerate}
 Then 
 \be
  \{\lambda\} = \sigma(B) \cap \{z:\quad |z| < \frac{\eps}{2} \}
 \ee
 with $|\lambda| \leq t \eps$.
 Denote by $\varphi$ a normalized solution of $B\varphi = \lambda\varphi$.
 If $\|A\psi\| \leq t \eps$, then there exists $a$ such that
 \be
  \|\psi - a \varphi\| \leq 16 t.
 \ee
\end{proposition}

\begin{proof}
 The first part of the statement follows from eigenvalues
 being Lipschitz in the perturbation. For the second part,
 write
 \[
  \psi = \spr{\varphi}{\psi} \varphi + \psi^{\perp}.
 \]
 Since $\psi^{\perp}$ is in the orthogonal complement
 of $\varphi$, we have that 
 \[
  \|( B - \lambda) \psi^{\perp}\| \geq \frac{\eps}{4} \|\psi^{\perp}\|.
 \]
 Hence, we obtain
 \[
  \|(B - \lambda) \psi\| = \|(B - \lambda) \psi^{\perp}\|
   \geq \frac{\eps}{4} \|\psi^{\perp}\|.
 \]
 Using that $\|(B - \lambda)\psi\| \leq 4 t \eps$,
 the claim follows.
\end{proof}

Define $A_u(v) = D(u,v) + T(u,v)$ for $u$ small enough.
It follows from the general theory of normal operators
that the eigenvalue $\lambda(v)$ of $A_u(v)$ satisfying
\be
 |\lambda(v_+^{k,\ell})| = O(u^2)
\ee
and $v \mapsto \lambda(v)$ is an analytic function
whose derivative is given by
\be
 \lambda'(v) = \spr{\psi}{\partial_v A_u(v) \psi}
\ee
where $\psi$ is any normalized solution of 
$(A_u(v) - \lambda(v)) \psi = 0$.

\begin{lemma}
 Let $v = O(u^2)$.
 We have that $\lambda'(v) = e(-\frac{k}{p}) + O(u)$.
\end{lemma}

\begin{proof}
 The previous proposition with test function $\psi$ given by
 \[
  \psi(n,m) = \begin{cases} 1,& (n,m) = (k,\ell);\\
   0,&\text{otherwise}\end{cases}
 \]
 shows this claim.
\end{proof}

\begin{proof}[Proof of Proposition~\ref{prop:asymvpm}]
 By the previous results, we can find $\tau = O(u^2)$
 such that
 \[
  \lambda(v_+^{k, \ell} + \tau) = 0
 \]
 is a simple eigenvalue of $A_u(v_+^{k, \ell} + \tau)$.
 
 Repeating the previous the previous considerations
 for $v_-^{k,\ell}$ finishes the proof.
\end{proof}

%%%%%%%%%%%%%%%%%%%%%%%%%%%%%%%%%%%%%%%%%%%%%%%%%%%%%%%%%%%%%%%%%%%%%
%
%
%

\section{The spectrum of two dimensional periodic Schr\"odinger operators}
\label{sec:disBS}

The next theorem is our discrete analog of the
Bethe--Sommerfeld conjecture.

\begin{theorem}\label{thm:BSConj}
 Let $p,q \geq 2$ be co-prime. Then there exists $\delta = \delta(p,q) > 0$
 such that for any $(p,q)$-periodic $V:\Z^2 \to \R$ with $\|V\|_{\infty} \leq \delta$,
 $\sigma(\Delta + V)$ is an interval.
\end{theorem}

\begin{proof}
 We have $\sigma(\Delta) = [-2d,2d]$.
 For any $E \in (-2d,2d)$ there exist
 some $\ell, k \in\mathbb{V}$ such that
 \[
  2 \cos\left(2\pi (k + \theta)\right) + 
  2 \cos\left(2\pi (\ell + \varphi)\right)
  = E
 \]
 for infinitely many $(\theta, \varphi) \in (0,\frac{1}{p})
 \times (0,\frac{1}{q})$. Furthermore, these
 are all points of increase.
 By Theorem~\ref{thm:coprimeevsimple}, at most
 finitely many of these do not correspond to
 simple eigenvalues. Hence, $E_j(\theta, \varphi) = E$
 is a simple eigenvalue and increasing.
 By Lemma~\ref{lem:incimpint}, every $E$ is
 thus in the interior of a band.
 This implies that the
 bands are overlapping. The claim then
 follows by Theorem~\ref{thm:bandoverlap}.
\end{proof}

The following example shows that it is necessary 
that at either $p$ or $q$ is odd for the conclusions
of the previous theorem hold. This begs the following
question

\begin{question}\label{ques:coprime}
 What is the optimal condition on $p$ and $q$ such
 that the conclusions of the previous theorem hold?
\end{question}

We now come to the counterexample with even periods.
Let $d\geq 1$, $\delta > 0$ and define a $2$-periodic potential
by
\be
 V_{\delta}(\ul{n}) = \begin{cases} \delta,& \sum_{j=1}^{d} n_j \mod{2} = 0;\\
 - \delta,&\text{otherwise}.\end{cases}
\ee
Clearly $V_{\delta} = \delta V_{1}$ and $\|V\|_{\infty} \leq \delta$.

\begin{theorem}\label{thm:ex2p}
 We have that
 \be
  \sigma(\Delta + V_{\delta}) \cap (-\delta, \delta) = 0.
 \ee
\end{theorem}

This theorem shows that given $p, q$ even. Then
then there exists a $(p,q)$-periodic potential $V$
with $\|V\|_{\infty}$ arbitrarily small such that
the spectrum of $\Delta + V$ contains a gap.
Unfortunately, this example is very specific 
and only allows to create one gap in the center
of spectrum, in order to construct limit-periodic
examples with Cantor spectrum, one would need a
better mechanism. This brings us to

\begin{question}\label{ques:gaps}
 Is there another way to open gaps?
\end{question}

The key step of the proof is the following lemma

\begin{lemma}
 For any $\psi \in \ell^2(\Z^d)$ and $\delta > 0$, we have that 
 \be
  \spr{\Delta \psi}{V_{\delta} \psi} = 0.
 \ee
\end{lemma}

\begin{proof}
 Write $\Delta = \sum_{j=1}^{d} \Delta_j$ with $\Delta_j \psi(\ul{n})
 = \psi(\ul{n}+b_j)+\psi(\ul{n}-b_j)$. By linearity, it clearly
 suffices to show that $\spr{\Delta_j \psi}{V_{\delta} \psi} = 0$
 for each $j$. Compute
 \[
  \spr{\Delta_j \psi}{V_{\delta} \psi}
  = \sum_{\ul{n}} \psi(\ul{n}) \psi(\ul{n}+b_j) (V_{\delta}(\ul{n})
   + V_{\delta}(\ul{n} + b_j)).
 \]
 Since $V_{\delta}(\ul{n}) + V_{\delta}(\ul{n} + b_j) = 0$,
 the claim follows.
\end{proof}

\begin{proof}[Proof of Theorem~\ref{thm:ex2p}]
 A computation using the last lemma shows
 \[
  \|(\Delta+V_{\delta})\psi\|^2 = \|\Delta\psi\|^2 + \delta^2 \|\psi\|^2
   \geq \delta^2 \|\psi\|^2
 \]
 for $\psi\in\ell^2(\Z^d)$. This implies the claim.
\end{proof}

%%%%%%%%%%%%%%%%%%%%%%%%%%%%%%%%%%%%%%%%%%%%%%%%%%%%%%%55
%
%
%

\section{Limit-periodic potentials}

We recall that a function $V:\Z^d\to\R$ is limit-period if
it is the limit of periodic functions. The following theorem
asserts the existence of limit-periodic potentials
whose spectrum is an interval. In fact it shows that
the spectrum of all limit-periodic potentials with suitable periods
that are sufficiently small in an appropriate sense
is an interval.

\begin{theorem}\label{thm:lpinterval}
 Let $p, q \geq 2$ be coprime. There exists a sequence of $\delta_t > 0$
 such that $\sum_{t=1}^{\infty} \delta_t < \infty$ and
 if $V_t:\Z^d\to\R$ is a sequence of $(p^t, q^t)$-periodic
 potentials satisfying $\|V_t\|_{\infty} \leq \delta_t$
 then the spectrum of
 \be
  \Delta + \sum_{t=1}^{\infty} V_t
 \ee
 is an interval.
\end{theorem}

An important question that this theorem leaves unanswered
is the qualitative behavior of the sequence $\delta_t > 0$.
In the continuous setting Karpeshina and Lee \cite{kl1},
\cite{kl2}, \cite{kl3} have shown that one can
take $\delta_t = C \exp(-2^{\eta t})$ for some constants
$C, \eta > 0$ to obtain that the spectrum contains
a semi-axis. 

I expect that using KAM-type techniques as employed by
Karpeshina and Lee, one should be able to obtain a similar
estimate in our setting. However, such a proof will be
much more involved then the one given here. The
estimates of Karpeshina and Lee would also allow
us to prove that the spectrum is purely absolutely continuous,
which is not possible using our methods.

\bigskip

We now begin the proof of Theorem~\ref{thm:lpinterval}.
We first note that if $p,q \geq 2$ are coprime, then also
$p^t, q^t$ are coprime for $t \geq 1$.

\begin{proposition}
 Let $p,q\geq 2$ be coprime and 
 $V:\Z^d\to\R$ be $(p^{t}, q^{t})$-periodic.
 Assume the bands of $\Delta + V$ viewed as a $(p^{t}, q^{t})$-periodic
 operator are overlapping, then for $s \geq t$ also the bands
 of $\Delta + V$ viewed as a $(p^{s}, q^{s})$-periodic operator
 are overlapping.
\end{proposition}

This proposition makes no claim over the size of the overlap.
In fact as $s \to \infty$, the size of the overlap goes
to $0$ as we have seen in Theorem~\ref{thm:bandsize}.

\begin{proof}
 Let $\sigma(\Delta + V) = [E_0, E_1]$. Let $E \in (E_0, E_1)$,
 by Theorem~\ref{thm:ptofincrease} there are infinitely many $(\varphi, \theta)$
 such that $E_j(\varphi, \theta) = E$ for some $j$ and it
 is increasing. By Theorem~\ref{thm:coprimeevsimple} at most finitely many of
 them are not simple eigenvalues. Hence, there is
 $(\varphi, \theta)$ and $j$ such that $E_j(\varphi, \theta) = E$
 is a simple eigenvalue and $E_j$ is increasing.
 By Lemma~\ref{lem:incimpint}, $E$ is in the interior of a band. 
 The claim follows.
\end{proof}

\begin{lemma}
 Let $p,q\geq 2$ be coprime and $t\geq 1$.
 Let $\mathcal{V}$ be a compact set of $(p^t,q^t)$-periodic
 potentials, such that for every $V \in \mathcal{V}$
 the bands of $\Delta + V$ viewed as a $(p^t, q^t)$-periodic
 operator are overlapping.

 Then there exists $\delta > 0$ such that for every
 $(p^{t+1}, q^{t+1})$-periodic $W$ with $\|W\|_{\infty} \leq \delta$,
 the bands of $\Delta + V + W$ are overlapping.
\end{lemma} 

\begin{proof}
 By the previous proposition, also the bands of
 $\Delta + V$ viewed as a $(p^{t+1}, q^{t+1})$-periodic
 operator are overlapping for $V \in \mathcal{V}$.
 Thus by Lemma~\ref{lem:uniformoverlap}, there exists
 $\ti{\delta} > 0$ such that for the bands of $\Delta + V$
 viewed as a $(p^{t+1}, q^{t+1})$-periodic
 operator are $\ti{\delta}$-overlapping for $V \in \mathcal{V}$.
 Take $\delta = \frac{1}{2}\ti{\delta}$ and the claim
 follows by Theorem~\ref{thm:bandoverlap}.
\end{proof}

\begin{proof}[Proof of Theorem~\ref{thm:lpinterval}]
 We have already seen that there exists $\delta_1 > 0$
 such that the bands of $\Delta + V_1$
 are overlapping for any $(p,q)$-periodic $V_1$
 with $\|V_1\|_{\infty} \leq \delta_1$.
 We denote the set of all such $V_1$ by $\mathcal{V}_1$,
 which is clearly compact. 

 By the previous lemma, there exists $\delta_2 > 0$,
 such that for all $V_1 \in \mathcal{V}_1$ and
 all $(p^2, q^2)$-periodic $V_2$ satisfying $\|V_2\|_{\infty} \leq \delta_2$,
 we have that the bands of $\Delta + V_1 + V_2$
 viewed as a $(p^2, q^2)$-periodic operator
 are overlapping. We denote the set of all
 such $V_2$ by $\mathcal{V}_2$. We also see that
 $\mathcal{V}_2$ and $\mathcal{V}_1 + \mathcal{V}_2$
 are compact.

 We can iterate this process to construct sets
 $\mathcal{V}_t$ consisting of all $(p^t, q^t)$-periodic
 $V_t$ with $\|V_t\|_{\infty} \leq \delta_t$
 for some sequence $\delta_t > 0$. By possibly
 making $\delta_t > 0$ smaller, we can assume that
 $\sum_{t=1}^{\infty} \delta_t$ converges. Furthermore,
 we will have that for every $T \geq 1$ and
 $V_t \in \mathcal{V}_t$ for $1\leq t\leq T$
 \[
  \sigma\left(\Delta + \sum_{t=1}^{T} V_t\right)
 \]
 is an interval. Since $\|\sum_{t=T}^{\infty} V_t\|_{\infty} \to 0$
 as $T\to\infty$, it follows that also
 \[
  \sigma\left(\Delta + \sum_{t=1}^{\infty} V_t\right)
 \]
 is an interval, which is the claim.
\end{proof}

%%%%%%%%%%%%%%%%%%%%%%%%%%%%%%%%%%%%%%%%%%%%%%%%%%%%%%%%%%%%%%%%%%%%%
%
%
%

\end{document}